\documentclass[a4paper,reqno]{amsart}
\usepackage{amssymb,amsmath,hyperref}
\newcommand{\C}{\mathbb{C}}
\newcommand{\R}{\mathbb{R}}
\newcommand{\Oo}{\mathcal{O}}
\newcommand\ko{\hspace{.1em}}

\newtheorem{theorem}{Theorem}[section]
\newtheorem{lemma}[theorem]{Lemma}
\theoremstyle{definition}
\newtheorem{definition}[theorem]{Definition}
 
\numberwithin{equation}{section}

\title{An elementary approach to the group law on elliptic curves}
\date{\today}
\author{Sander Zwegers}
\address{Department of Mathematics and Computer Science, University of Cologne, Weyertal 86--90, 50931 Cologne, Germany}
\email{szwegers@uni-koeln.de}
\begin{document}

\begin{abstract}
We revisit the group structure on elliptic curves and give a simple and elementary proof of the associativity of the addition.
We do this by providing an explicit formula for the sum of three points, only using the explicit definition of the group structure.
In the process we find a nice geometric interpretation of the sum of three points on the elliptic curve.
\end{abstract}

\maketitle

\section{Introduction and statement of results}

We consider the elliptic curve
\[ \{ (x,y)\in K^2 \mid y^2=x^3+ax+b\} \cup \{ \Oo\}\]
in (reduced) Weierstra\ss\ form.
Here $K$ is a field with characteristic $\neq 2,3$ and $\Oo$ is the point at infinity.
We assume $4a^3+27b^2\neq 0$, which ensures that the polynomial $x^3+ax+b$ has no multiple zeroes. 
We equip the elliptic curve with the usual group law.

\begin{lemma}\label{lem:add}
Let $P_1=(x_1,y_1)$ and $P_2=(x_2,y_2)$ be in $E\smallsetminus\{\Oo\}$ with $(x_2,y_2)\neq (x_1,-y_1)$.
We set
\[\begin{split}
x_3:=-x_1-x_2+\alpha^2\qquad\text{and}\qquad  y_3:=& -y_1-\alpha(x_3-x_1)\\
=&-y_2-\alpha(x_3-x_2),
\end{split}\]
where
\[ \alpha =\alpha(P_1,P_2):= \begin{cases} \frac{y_2-y_1}{x_2-x_1},&\text{if}\ x_1\neq x_2,\\ \frac{3x_1^2+a}{2y_1}, & \text{if}\ x_1=x_2.\end{cases}\]
Then we have $(x_3,y_3)\in E\smallsetminus\{\Oo\}$.
\end{lemma}

\begin{definition}\label{def:ell}
We define the addition $+:E\times E\longrightarrow E$ by\\
(a) $P+\Oo=\Oo+P=P$ for all $P\in E$;\\
(b) $(x,y)+(x,-y)=\Oo$ for $(x,y)\in E\smallsetminus\{\Oo\}$;\\
(c) $(x_1,y_1)+(x_2,y_2)=(x_3,y_3)$ for $(x_1,y_1),(x_2,y_2)\in E\smallsetminus\{\Oo\}$ with $(x_2,y_2)\neq (x_1,-y_1)$, where $x_3$ and $y_3$ are as in Lemma \ref{lem:add}.
\end{definition}

This describes the usual addition on the elliptic curve $E$.
The next result is well known.

\begin{theorem}
$(E,+)$ is an abelian group with neutral element $\Oo$.
\end{theorem}

From the definition it is clear that the addition is commutative, that $\Oo$ is the neutral element, that the inverse of $P=(x,y)\in E\smallsetminus\{\Oo\}$ is given by $-P:=(x,-y)$ and that $-\Oo=\Oo$.
The only non-trivial part is the associativity.
One approach to this is to concentrate on elliptic curves over $\C$ and use the theory of elliptic functions, in particular the Weierstra\ss\ $\wp$-function.
See for example \cite{koblitz} or \cite{lang} for this approach.
Another possibility is to use somewhat advanced algebraic geometry, for example the Riemann--Roch theorem, the Cayley--Bacharach theorem and/or B\'ezout's theorem.
We refer the reader to \cite{fulton,husem,silverman,washington} for various geometric proofs.
In \cite{friedl} a completely elementary proof of the group law is given only using the explicit formulas in the definition of the addition.
However, the proof of the associativity requires some advanced computer-aided arguments.
The author notes:
``\ldots it turns out that the explicit computations in the proof are very hard.
The verification of some identities took several hours on a modern computer; this proof could not have been carried out before the 1980's.''

Here we give an elementary proof of the associativity, not relying on the use of computer calculations.

\begin{theorem}\label{theo:main}
For all $P_1,P_2,P_3\in E$ we have $(P_1+P_2)+P_3=P_1+(P_2+P_3)$.
\end{theorem}

For the proof we give an explicit formula for the coordinates of $(P_1+P_2)+P_3$ in the generic case.

\begin{theorem}\label{theo:sum}
Let $P_1,P_2,P_3\in E\smallsetminus\{\Oo\}$ with $P_3\neq \pm P_1$, $P_2\neq -P_1$, $P_2\neq -P_3$ and $P_1+P_2\neq -P_3$.
Then the coordinates of $P_4:=(P_1+P_2)+P_3$ are given by
\[ x_4= -x_1+x_2-x_3 + \frac1{c_2^2} - 2 \ko\frac{c_1}{c_2},\qquad y_4=-y_2-(x_4-x_2)(c_1+x_4c_2),\]
where
\[ c_1:=\frac{x_1\ko\alpha(P_2,P_3)-x_3\ko\alpha(P_1,P_2)}{x_1-x_3} ,\qquad c_2:= \frac{\alpha(P_1,P_2)-\alpha(P_2,P_3)}{x_1-x_3}.\]
\end{theorem}

The point is that these formulas are symmetric in $P_1$ and $P_3$, which  proves Theorem \ref{theo:main} in the generic case.
Note that a similar approach is used in \cite{spandaw}.
However, the formulas there involve polynomials of high degree in the coordinates of $P_1,P_2,P_3$ with thousands of terms.
Checking the associativity therefor requires a computer.
Our proof relies mostly on simple and elementary algebraic transformations.

In \cite{enge} the author remarks the following about a potential proof with a brute force computation, using the explicit algebraic formulas for adding points:
``..., the proof does not reveal anything about the underlying algebraic and geometric structures and is not only extremely tedious, but also extremely uninstructive.
This seems to have deterred  most authors, for, to my knowledge, this approach cannot be found in any publication.''
Although a bit tedious, our proof is fairly short and easy, and more interestingly, we actually find a nice geometric interpretation of the sum of three points:
we draw the parabola through the points $P_1$, $P_2$ and $P_3$; then the parabola has a fourth point of intersection with the elliptic curve, which is $-(P_1+P_2+P_3)$ (see Section \ref{sec:par} for details).

\section{Proof of Lemma \ref{lem:add} and Theorem \ref{theo:sum}}\label{sec:firstproof}

We start by proving Lemma \ref{lem:add}, which ensures that the addition on $E$ is well-defined.
First we note that if $x_1=x_2$, then we have $y_1^2=y_2^2$, so $y_1=\pm y_2$.
Since we exclude $y_2=-y_1$ in the lemma, we must have $y_1=y_2$ (so $P_1=P_2$).
Therefor $y_1\neq 0$ (otherwise $y_2=y_1=0=-y_1$), so we get $\alpha\in K$ and $(x_3,y_3)\in K^2$.
We now consider the polynomial
\[ f(x):=x^3+ax+b-\bigl(y_1+\alpha(x-x_1)\bigr)^2\in K[x].\]
Since $(x_1,y_1)$ is on the curve, we have $f(x_1)=x_1^3+ax_1+b-y_1^2=0$.
Hence $x_1$ is a zero of $f(x)$.
For the case $x_2\neq x_1$ we have $y_1+\alpha(x_2-x_1) = y_2$.
This gives $f(x_2)=x_2^3+ax_2+b-y_2^2=0$, so $x_2$ is another zero.
If we have $x_2= x_1$, then we consider the formal derivative
\[f'(x)=3x^2+a-2\alpha\bigl(y_1+\alpha(x-x_1)\bigr)\in K[x]\]
of $f(x)$.
We immediately see that $x_1$ is a zero of $f'(x)$, so $x_1$ is a double zero of $f(x)$.
In both cases we can write $f(x)$ as $f(x)=(x-x_1)(x-x_2)g(x)$, where $g(x)\in K[x]$.
Since $f(x)$ has degree 3 and leading coefficient 1, we can write $g(x)$ as $g(x)=x-x_0$ with $x_0\in K$.
We thus have
\[x^3+ax+b-(y_1+\alpha(x-x_1))^2=(x-x_1)(x-x_2)(x-x_0).\]
Comparing the coefficients of $x^2$ on the left and on the right we get $x_0+x_1+x_2=\alpha^2$, and so $x_0=x_3$ and $-(y_1+\alpha(x_0-x_1))=y_3$.
Since $x_0$ is a zero of $f(x)$, the point $(x_0,-(y_1+\alpha(x_0-x_1)))=(x_3,y_3)$ lies on the curve.
Hence $(x_3,y_3)\in K^2$ satisfies the equation $y^2=x^3+ax+b$.
Further, we notice that
\[x^3+ax+b-\bigl(y_1+\alpha(x-x_1)\bigr)^2=(x-x_1)(x-x_2)(x-x_3).\]

\begin{lemma}\label{lem:null}
Let $P_1,P_2\in E$ with $P_1+P_2=P_1$.
Then $P_2=\Oo$.
\end{lemma}

\begin{proof}
If $P_1=\Oo$, then we find $P_2=\Oo+P_2=P_1+P_2=P_1=\Oo$ as desired.
If $P_2=-P_1$, then we get $P_1=P_1+P_2=\Oo$ and again $P_2=-P_1=-\Oo=\Oo$.
Hence we must show that there exist no $P_1,P_2\in E\smallsetminus\{\Oo\}$ with $P_2\neq -P_1$ and $P_1+P_2=P_1$.
We assume that $P_2\neq -P_1$ and $P_1+P_2=P_1$ hold and use Lemma \ref{lem:add}:
we have $(x_1,y_1)=(x_3,y_3)$ and so we find $y_1=y_3=-y_1-\alpha(x_3-x_1)=-y_1$, $y_1=0$ and
\[ x^3+ax+b-\alpha^2(x-x_1)^2=(x-x_1)^2(x-x_2).\]
This means that $x_1$ is a double zero of $x^3+ax+b$, which contradicts our assumption $4a^3+27b^2\neq 0$.
\end{proof}

We now turn our attention to the proof of Theorem \ref{theo:sum}.
We set $(\widetilde x,\widetilde y):=P_1+P_2=(x_1,y_1)+(x_2,y_2)$ and $(x_4,y_4):= (\widetilde x,\widetilde y)+(x_3,y_3)=(P_1+P_2)+P_3$.
Hence we have
\[ \widetilde x =-x_1-x_2+\alpha^2,\qquad \widetilde y = -y_1-\alpha(\widetilde x-x_1),\]
where $\alpha:= \alpha(P_1,P_2)$, and
\[ x_4=-\widetilde x-x_3+\widetilde \alpha^2, \qquad y_4=-y_3-\widetilde\alpha(x_4-x_3),\]
where $\widetilde\alpha := \alpha(P_1+P_2,P_3)$.

\begin{lemma}\label{lem:c2}
Let $P_1,P_2,P_3\in E\smallsetminus\{\Oo\}$ with $P_3\neq \pm P_1$, $P_2\neq -P_1$, $P_2\neq -P_3$  and $P_1+P_2\neq -P_3$.
Further, let $c_2$ be as in Theorem \ref{theo:sum}.
Then $c_2(\alpha+\widetilde\alpha)=1$.
\end{lemma}

\begin{proof}
As in the proof of Lemma \ref{lem:add} we have 
\begin{equation}\label{eq:prodsum2}
x^3+ax+b-\bigl(y_1+\alpha(x-x_1)\bigr)^2=(x-x_1)(x-x_2)(x-\widetilde x).
\end{equation}
If $P_2\neq P_3$, then we have
\[\begin{split}
c_2&=\frac{\alpha(P_1,P_2)-\alpha(P_2,P_3)}{x_1-x_3}= \frac{\alpha - \frac{y_2-y_3}{x_2-x_3}}{x_1-x_3}\\
&= \frac{y_3-y_2+\alpha(x_2-x_3)}{(x_1-x_3)(x_2-x_3)}= \frac{y_3-y_1-\alpha(x_3-x_1)}{(x_3-x_1)(x_3-x_2)}.
\end{split}\]
If $P_1+P_2\neq P_3$, then we have
\[ \alpha+\widetilde\alpha = \alpha+ \frac{y_3-\widetilde y}{x_3-\widetilde x} = \alpha+\frac{y_3+y_1+\alpha(\widetilde x-x_1)}{x_3-\widetilde x}=\frac{y_3+y_1+\alpha( x_3-x_1)}{x_3-\widetilde x}.\]
We now consider 4 possible cases.
For the case $P_2\neq P_3$ and $P_1+P_2\neq P_3$ we directly get
\[ c_2 (\alpha+\widetilde\alpha) = \frac{y_3^2 -\bigl(y_1+\alpha(x_3-x_1)\bigr)^2}{(x_3-x_1)(x_3-x_2)(x_3-\widetilde x)}=1,\]
where for the last equality we used \eqref{eq:prodsum2} with $x=x_3$ and $y_3^2=x_3^3+ax_3+b$.\\
For the case $P_2\neq P_3$ and $P_1+P_2= P_3$ we have $(\widetilde x,\widetilde y)=(x_3,y_3)$, so $y_3=-y_1-\alpha(x_3-x_1)$ and
\[ c_2= \frac{y_3-y_1-\alpha(x_3-x_1)}{(x_3-x_1)(x_3-x_2)}=\frac{2y_3}{(x_3-x_1)(x_3-x_2)}.\]
In \eqref{eq:prodsum2} we take the derivative with respect to $x$ and set $x=\widetilde x=x_3$ to get
\[3x_3^2+a+2\alpha y_3=(x_3-x_1)(x_3-x_2),\qquad 
\alpha+\widetilde\alpha =\frac{(x_3 -x_1)(x_3-x_2)}{2y_3}\]
(as remarked in the proof of Lemma \ref{lem:add} we have $y_3\neq 0$) and $c_2 (\alpha+\widetilde\alpha) =1$.\\
For the case $P_2=P_3$ and $P_1+P_2\neq P_3$ we have $(x_2,y_2)=(x_3,y_3)$, so $y_3=y_1+\alpha(x_3-x_1)$ and
\[ \alpha+\widetilde\alpha= \frac{y_3+y_1+\alpha( x_3-x_1)}{x_3-\widetilde x}=\frac{2y_3}{x_3-\widetilde x}.\]
Again, we take the derivative in \eqref{eq:prodsum2} and set $x=x_2=x_3$ to get
\[ 3x_3^2+a-2\alpha y_3=(x_3-x_1)(x_3-\widetilde x).\]
Hence we have
\[\alpha(P_2,P_3) -\alpha(P_1,P_2) = \frac{(x_3-x_1)(x_3-\widetilde x)}{2y_2},\qquad c_2= \frac{x_3-\widetilde x}{2y_3}\]
($y_2=y_3\neq 0$) and $c_2 (\alpha+\widetilde\alpha) =1$.\\
Finally, according to Lemma \ref{lem:null} the case $P_2= P_3$ and $P_1+P_2=P_3$ can not occur.
\end{proof}

\begin{proof}[Proof of Theorem \ref{theo:sum}]
From the definition of $c_1$ and $c_2$ we immediately get
\[ c_1+x_1\ko c_2=\alpha(P_1,P_2)=\alpha\qquad \text{and}\qquad c_1+x_3 \ko c_2=\alpha(P_2,P_3).\]
Further, Lemma \ref{lem:c2} gives $\widetilde \alpha =1/c_2 -\alpha$.
Hence we get
\[ \begin{split}
x_4&= -\widetilde x-x_3+\widetilde \alpha^2 = x_1+x_2-x_3+\widetilde\alpha^2-\alpha^2= x_1+x_2-x_3 +\frac1{c_2^2} -2 \frac\alpha{c_2}\\
&= -x_1+x_2-x_3 +\frac1{c_2^2} -2 \ko\frac{\alpha-x_1\ko c_2}{c_2}=  -x_1+x_2-x_3 +\frac1{c_2^2} -2\ko \frac{c_1}{c_2}.
\end{split}\]
Also we have
\[ y_3-y_2 =\alpha(P_2,P_3)(x_3-x_2) = (x_3-x_2)(c_1+x_3c_2),\]
so
\[ \begin{split}
y_3-y_2 -(x_4-x_2)(c_1+x_4c_2) &= (x_3-x_2)(c_1+x_3c_2)-(x_4-x_2)(c_1+x_4c_2) \\
&= -(x_4-x_3)\bigl(c_1+(-x_2+x_3+x_4)\ko c_2\bigr).
\end{split}\]
With $y_4=-y_3-\widetilde\alpha(x_4-x_3)$ and $x_4=x_1+x_2-x_3+\widetilde\alpha^2-\alpha^2$ we then obtain
\[\begin{split}
y_4+y_2+(&x_4-x_2)(c_1+x_4c_2)= (x_4-x_3)\bigl( -\widetilde\alpha +c_1 +(-x_2+x_3+x_4)\ko c_2\bigr)\\
&= (x_4-x_3)\bigl( -\widetilde\alpha +c_1+x_1c_2+(\widetilde\alpha^2-\alpha^2)\ko c_2\bigr)\\
&= (x_4-x_3)\bigl( -\widetilde\alpha +\alpha+(\widetilde\alpha^2-\alpha^2)\ko c_2\bigr)= (x_4-x_3)(\alpha-\widetilde\alpha) \bigl(1-(\alpha+\widetilde\alpha)\ko c_2\bigr)=0
\end{split}\]
and $y_4=-y_2-(x_4-x_2)(c_1+x_4c_2)$.
\end{proof}

\section{Proof of Theorem \ref{theo:main}}

\begin{lemma}\label{lem:121}
For all $P,P_1,P_2\in E$ we have:\\
(a) $-(-P)=P$;\\
(b) $-(P_1+P_2)=(-P_1)+(-P_2)$;\\
(c) $(P_1+P_2)+(-P_1)=P_1+(P_2+(-P_1))=P_2$;\\
(d) if $P_1+P_2=\Oo$, then $P_2=-P_1$.
\end{lemma}

\begin{proof}
(a) For $P=(x,y)\in E\smallsetminus\{\Oo\}$ we have $-P=(x,-y)\in E\smallsetminus\{\Oo\}$, so $-(-P)=(x,y)=P$.
From $-\Oo=\Oo$ we immediately get $-(-\Oo)=\Oo$.\\
(b) If $P_1$ or $P_2$ equals $\Oo$, then the result follows directly:
\[ \begin{split}
P_1&=\Oo:\qquad -(P_1+P_2)=-(\Oo+P_2)=-P_2=\Oo+(-P_2)=(-\Oo)+(-P_2)=(-P_1)+(-P_2)\\
P_2&=\Oo:\qquad -(P_1+P_2)=-(P_1+\Oo)=-P_1=(-P_1)+\Oo= (-P_1)+(-\Oo)=(-P_1)+(-P_2)
\end{split}\]
Now let $P_1,P_2\in E\smallsetminus\{\Oo\}$.
If $P_2=-P_1$, then $-P_2=-(-P_1)=P_1$, so
\[ -(P_1+P_2)=-(P_1+(-P_1))=-\Oo=\Oo=(-P_1)+P_1=(-P_1)+(-P_2).\]
Now assume $P_2\neq -P_1$.
By definition we have
\[P_1+P_2=(-x_1-x_2+\alpha(P_1,P_2)^2,-y_1-\alpha(P_1,P_2)(x_3-x_1)).\]
Further, we have $-P_2\neq -(-P_1)$ and $-P_1=(x_1,-y_1)$, $-P_2=(x_2,-y_2)$, so $\alpha(-P_1,-P_2)=-\alpha(P_1,P_2)$ and
\[(-P_1)+(-P_2)= (-x_1-x_2+\alpha(P_1,P_2)^2,y_1+\alpha(P_1,P_2)(x_3-x_1))=-(P_1+P_2).\]
(c) We first show $(P_1+P_2)+(-P_1)=P_2$.
For $P_1=\Oo$ or $P_2=\Oo$ the result follows immediately:
$(\Oo+P_2)+(-\Oo)= P_2+\Oo=P_2$ and $(P_1+\Oo)+(-P_1)=P_1+(-P_1)=\Oo$.
Now let $P_1,P_2\in E\smallsetminus\{\Oo\}$.
We first consider the case that (b) in Definition \ref{def:ell} applies for one of the additions.
If $P_2=-P_1$, then $P_1+P_2=\Oo$ and so $(P_1+P_2)+(-P_1)=\Oo+P_2=P_2$.
On the other hand, the case $P_1+P_2=P_1$ doesn't occur, because it would contradict Lemma \ref{lem:null}.
Now assume that $P_2\neq -P_1$ and $P_1+P_2\neq P_1$ hold.
With $P_3=-P_1$ the notation as in Section \ref{sec:firstproof} applies.
We then have $P_2\neq P_3$ and $P_1+P_2\neq -P_3$.
If $P_1+P_2\neq P_3$, then (as in the proof of Lemma \ref{lem:c2})
\[ \alpha+\widetilde\alpha=\frac{y_3+y_1+\alpha( x_3-x_1)}{x_3-\widetilde x},\]
which equals 0, since $x_3=x_1$ and $y_3=-y_1$.
If $P_1+P_2= P_3$, then (again as in the proof of Lemma \ref{lem:c2})
\[\alpha+\widetilde\alpha =\frac{(x_3 -x_1)(x_3-x_2)}{2y_3},\]
which again equals 0.
In both cases we have shown $\widetilde \alpha=-\alpha$, so we get
\[ \begin{split}
x_4&=-\widetilde x -x_3+\widetilde \alpha^2=-\widetilde x-x_1+\alpha^2=x_2,\\
y_4&=-y_3-\widetilde \alpha(x_4-x_3)=y_1+\alpha(x_2-x_1)=y_2.
\end{split}\]
This shows $P_4=P_2$ and so $(P_1+P_2)+(-P_1)=(P_1+P_2)+P_3=P_4=P_2$, as desired.
Replacing $P_1$ by $-P_1$ and using (a) and the commutativity of the addition we then also get
\[ P_2= ((-P_1)+P_2)+(-(-P_1))=((-P_1)+P_2)+P_1=P_1+(P_2+(-P_1)).\]
(d) If $P_1+P_2=\Oo$, then $P_2=(P_1+P_2)+(-P_1)=\Oo+(-P_1)=-P_1$.
\end{proof}

\begin{proof}[Proof of Theorem \ref{theo:main}]
If $P_1$, $P_2$ or $P_3$ equals $\Oo$, then the result follows directly:
\[ \begin{split}
P_1&=\Oo:\qquad (P_1+P_2)+P_3=(\Oo+P_2)+P_3=P_2+P_3=\Oo+(P_2+P_3)= P_1+(P_2+P_3)\\
P_2&=\Oo:\qquad (P_1+P_2)+P_3=(P_1+\Oo)+P_3=P_1+P_3=P_1+(\Oo+P_3)= P_1+(P_2+P_3)\\
P_3&=\Oo:\qquad (P_1+P_2)+P_3=(P_1+P_2)+\Oo=P_1+P_2=P_1+(P_2+\Oo)= P_1+(P_2+P_3)
\end{split}\]
If $P_3=P_1$, then the result follows from the commutativity of the addition:
\[ (P_1+P_2)+P_3=(P_1+P_2)+P_1=P_1+(P_2+P_1)= P_1+(P_2+P_3)\]
For the case $P_3=-P_1$ the result follows from Lemma \ref{lem:121}(c):
\[ (P_1+P_2)+P_3= (P_1+P_2)+(-P_1)=P_1+(P_2+(-P_1)) = P_1+(P_2+P_3)\]
For $P_2=-P_1$ or $P_2=-P_3$ we also use Lemma \ref{lem:121}(c):
\[ \begin{split}
P_2=-P_1:\qquad (P_1+P_2)+P_3&=(P_1+(-P_1))+P_3=\Oo+P_3=P_3\\
&=P_1+(P_3+(-P_1))=P_1+((-P_1)+P_3)= P_1+(P_2+P_3)\\
P_2=-P_3:\qquad (P_1+P_2)+P_3&=(P_1+(-P_3))+P_3=P_3+(P_1+(-P_3))=P_1\\
&=P_1+\Oo=P_1+((-P_3)+P_3)= P_1+(P_2+P_3)
\end{split}\]
Next we consider the case $P_1+P_2=-P_3$.
With Lemma \ref{lem:121} we find
\[P_2+P_3=P_2+(-(-P_3))=P_2+(-(P_1+P_2))=P_2+((-P_1)+(-P_2))=-P_1\]
and so
\[(P_1+P_2)+P_3=(-P_3)+P_3=\Oo=P_1+(-P_1)=P_1+(P_2+P_3).\]
For the final case we assume $P_1,P_2,P_3\in E\smallsetminus\{\Oo\}$ with $P_3\neq \pm P_1$, $P_2\neq -P_1$, $P_2\neq -P_3$  and $P_1+P_2\neq -P_3$, which are exactly the conditions of Theorem \ref{theo:sum}.
From the previous case we get that $P_1+P_2\neq -P_3$ and $P_2+P_3\neq -P_1$ are equivalent (we have shown that $P_1+P_2=-P_3$ implies $P_2+P_3=-P_1$; the implication in the other direction follows by interchanging $P_1$ and $P_3$).
Hence the conditions of the theorem are also satisfied if we permute $P_1$ and $P_3$.
Since the formulas for $P_4$ are directly seen to be invariant under this permutation, we thus get $(P_1+P_2)+P_3= (P_3+P_2)+P_1=P_1+(P_2+P_3)$.
\end{proof}

\section{Geometric interpretation}\label{sec:par}

For $K=\R$, we have a well known geometric interpretation of the addition of two points:
let $P_1,P_2\in E\smallsetminus\{\Oo\}$ with $P_2\neq -P_1$.
We consider the line $\ell:y=y_1+\alpha (x-x_1)$ through the points $P_1$ and $P_2$.
If we have $P_2=P_1$, then this line is tangent to the elliptic curve in the point $P_1$. 
This property uniquely determines $\ell$.
As we have seen in the proof of Lemma \ref{lem:add} the line $\ell$ has a third point of intersection with the elliptic curve (counting multiplicities), which is exactly the point $(x_3,-y_3)=-(P_1+P_2)$.
If $P_2=-P_1$, then the line through $P_1$ and $P_2$ is a vertical line and the third point of intersection is the point $\Oo$ at infinity, which again equals $-(P_1+P_2)$.

Here we give a similar geometric interpretation of the addition of three points $P_1,P_2,P_3$:
we start with the generic case $P_i\neq \pm P_j$ for all $i,j$.
This implies that $x_1,x_2,x_3$ are pairwise distinct.
If we further assume $P_1+P_2\neq -P_3$, then the three points are not collinear and hence there is a unique parabola through these points.
We claim that it is given by the quadratic polynomial $p(x)=y_2+(x-x_2)(c_1+xc_2)$ (with $c_1,c_2$ as in Theorem \ref{theo:sum}):
we have $p(x_2)=y_2$ and
\[\begin{split}
p(x_1)&=y_2+(x_1-x_2)(c_1+x_1c_2)=y_2+\alpha(x_1-x_2)=y_1,\\
p(x_3)&=y_2+(x_3-x_2)(c_1+x_3c_2)=y_2+\alpha(P_2,P_3)(x_3-x_2)=y_3.
\end{split}\]
To find the points of intersection with the elliptic curve we have to solve $p(x)^2=x^3+ax+b$, which is a polynomial equation of degree 4 in $x$.
By construction $x_1$, $x_2$ and $x_3$ are solutions, so there exist $C,x_0\in \R$ with
\[f(x):=(y_2+(x-x_2)(c_1+xc_2))^2-(x^3+ax+b)= C(x-x_1)(x-x_2)(x-x_3)(x-x_0).\]
Comparing the coefficients of $x^4$ and $x^3$ on the left and on the right we find $c_2^2=C$ and $2c_2(c_1-c_2x_2)-1=-C(x_1+x_2+x_3+x_0)$ and so
\[ x_0=-x_1-x_2-x_3-\frac2{c_2}(c_1-c_2x_2)+\frac1{c_2^2}=-x_1+x_2-x_3+\frac1{c_2^2}-2\ko\frac{c_1}{c_2}=x_4,\]
(by Theorem \ref{theo:sum}).
Hence the  fourth point of intersection is $(x_4,p(x_4))=(x_4,-y_4)=-(P_1+P_2+P_3)$ (again using Theorem \ref{theo:sum}).
For the case $P_1=P_2\neq P_3$ we can basically use the same argument, we just have to account for multiplicities:
we have $p'(x_1)=c_1+x_1c_2=\alpha$, so the parabola and the elliptic curve have the same tangent at $P_1=P_2$.
Again, this uniquely determines the polynomial $p(x)$.
Further, $f'(x_1)=2\alpha y_1-(3x_1^2+a)=0$, so $x_1$ is a double zero of $f(x)$ and we have
\[(y_2+(x-x_2)(c_1+xc_2))^2-(x^3+ax+b)= C(x-x_1)^2(x-x_3)(x-x_0).\]
As before the fourth point of intersection is $(x_4,-y_4)=-(P_1+P_2+P_3)$.
With the associativity and a permutation of the indices this covers all cases where exactly two points are equal.

If for example $P_2=-P_1$, then $P_1$ and $P_2$ are on a vertical line and we have $-(P_1+P_2+P_3)=-P_3$, which is again the fourth point of intersection, where the parabola through $P_1$, $P_2$ and $P_3$ is understood to consist of two vertical lines.

Finally, we briefly discuss the case $P_1=P_2=P_3$.
The polynomial $p(x)=y_1+(x-x_1)(c_1+xc_2)$ with
\[ c_2=\frac{3x_1-\alpha^2}{2y_1}, \qquad c_1=\alpha -c_2 x_1\]
fits the elliptic curve in the point $P_1$ up to the second order derivative.
With the same computation as before we find that the fourth point of intersection is $(x_4,p(x_4))$, with
\[ x_4=-x_1+\frac1{c_2^2}-2\ko\frac{c_1}{c_2}.\]
By a calculation similar to the one in Theorem \ref{theo:sum} (or seeing this as limit of the case $P_1=P_2\neq P_3$), we can show that this equals $-3P_1$.
We leave the details to the reader.

\end{document}